\documentclass{amsart}
\usepackage{amsmath,amsthm,amssymb}
\usepackage{graphicx}
\usepackage{tikz-cd}
\usepackage{paralist}
\usepackage{environ}
\usepackage{xcolor}
\usepackage{hyperref}
\usepackage{centernot}

\usepackage{marvosym}
\usepackage{xcolor,cancel}

\makeatletter
\DeclareFontFamily{U}{tipa}{}
\DeclareFontShape{U}{tipa}{bx}{n}{<->tipabx10}{}
\newcommand{\arc@char}{{\usefont{U}{tipa}{bx}{n}\symbol{62}}}%

\newcommand{\arc}[1]{\mathpalette\arc@arc{#1}}

\newcommand{\arc@arc}[2]{%
  \sbox0{$\m@th#1#2$}%
  \vbox{
    \hbox{\resizebox{\wd0}{\height}{\arc@char}}
    \nointerlineskip
    \box0
  }%
}
\makeatother

\makeatletter
\newcommand{\doublewedge}{\big@doubleop{\wedge}}
\newcommand{\big@doubleop}[1]{%
  \DOTSB\mathop{\mathpalette\big@doubleop@aux{#1}}\slimits@
}

\newcommand\big@doubleop@aux[2]{%
  \sbox\z@{$\m@th#1#2$}%
  \makebox[1.35\wd\z@][s]{$\m@th#1#2\hss#2$}%
}

\newcommand{\abs}[1]{\left|#1\right|}     
\newcommand{\cl}{\mbox{cl}}  
\newcommand{\Int}{\mbox{int}} 
\newcommand{\bdy}{\mbox{bdy}} 
\newcommand{\rb}{\mbox{rb}} 
\newcommand{\crb}{\mathcal{R}b} 
\newcommand{\cyc}{\mbox{cyc}} 

\newcommand{\near}{\delta} 
\newcommand{\dcap}{\mathop{\cap}\limits_{\Phi}} 
\newcommand{\dnear}{\delta_{\Phi}} 
\newcommand{\norm}[1]{\left\|#1\right\|}  

\usepackage{pstricks,pst-text,pst-grad,pst-node,pst-3dplot,pstricks-add,pst-poly,pst-coil,pst-bar,pst-all,pst-plot,pst-2dplot} 
\usepackage{pst-fun,pst-blur}
\usepackage{pst-all}
\usepackage{subfigure}
\renewcommand{\thesubfigure}{\thefigure.\arabic{subfigure}}
\makeatletter
\renewcommand{\p@subfigure}{}
\renewcommand{\@thesubfigure}{\thesubfigure:\hskip\subfiglabelskip}
\makeatother

\theoremstyle{plain}
\newtheorem{theorem}{Theorem}
\newtheorem{lemma}{Lemma}
\newtheorem{remark}{Remark}
\newtheorem{definition}{Definition}

\newtheorem{example}{Example}
\newtheorem{corollary}{Corollary}

\begin{document}

\title{Descriptive Fixed Set Properties\\ for Ribbon Complexes}

\author[J.F. Peters]{J.F. Peters}
\address{
Computational Intelligence Laboratory,
University of Manitoba, WPG, MB, R3T 5V6, Canada and
Department of Mathematics, Faculty of Arts and Sciences, Ad\.{i}yaman University, 02040 Ad\.{i}yaman, Turkey,
}
\email{james.peters3@umanitoba.ca}
\thanks{The research has been supported by the Natural Sciences \&
Engineering Research Council of Canada (NSERC) discovery grant 185986 
and Instituto Nazionale di Alta Matematica (INdAM) Francesco Severi, Gruppo Nazionale per le Strutture Algebriche, Geometriche e Loro Applicazioni grant 9 920160 000362, n.prot U 2016/000036 and Scientific and Technological Research Council of Turkey (T\"{U}B\.{I}TAK) Scientific Human
Resources Development (BIDEB) under grant no: 2221-1059B211301223.}
\author[T. Vergili]{T. Vergili}
\address{
Department of Mathematics, Karadeniz Technical University, Trabzon, Turkey,
}
\email{tane.vergili@ktu.edu.tr}


\date{}

\dedicatory{Dedicated to L.E.J. Brouwer and Mahlon M. Day}

\begin{abstract}
This article introduces descriptive fixed sets and their properties in descriptive proximity spaces viewed in the context of planar ribbon complexes.  These fixed sets are a byproduct of descriptive proximally continuous maps that spawn fixed subsets, eventual fixed subsets and almost fixed subsets of the maps.  For descriptive continuous map $f$ on a descriptive proximity space $X$, a subset $A$ of $X$ is fixed, provided the description of $f(A)$ matches the desription of $A$.  In terms ribbon complexes in a CW space, an Abelian group representation of a ribbon is Day-amenable and each amenable ribbon has a fixed point.  A main result in this paper is that if $h$ is a proximal descriptive conjugacy between maps $f,g$, then if $A$ is an [ordinary, eventual, almost] descriptively fixed subset of $f$, then $h(A)$ is a descriptively fixed subset of $g$.
\end{abstract}
\maketitle
\tableofcontents

\section{Introduction}
This paper introduces descriptive fixed set properties that are a natural outcome of descriptive proximally continuous set-valued maps.  Results given here for such maps spring from fundamental results for fixed points given by L.E.J. Brouwer~\cite{Brouwer1911fixedPoints}   

\begin{theorem}\label{thm:Brouwer}{\rm Brouwer Fixed Point Theorem~\cite[\S 4.7, p. 194]{Spanier1966AlgTopology}}$\mbox{}$\\
Every continuous map from $\mathbb{R}^n$ to itself has a fixed point.
\end{theorem}

\noindent and given by M.M. Day~\cite{Day1957amenableSemigroups,Day1961amenableSemigroups}, which carry over in terms of finite group representations of ribbons~\cite{Peters2020AMSBullRibbonComplexes} containing intersecting cycles or bridge edges between ribbon cycles.

\begin{theorem}\label{thm:Day}{\rm Day Abelian Group [Semigroup] Theorem~\cite[p. 516]{Day1957amenableSemigroups,Day1961amenableSemigroups}}$\mbox{}$\\
Every finite group is amenable.
\end{theorem}

A direct consequence of Theorem~\ref{thm:Day} is that each amenable ribbon has a fixed point.

A number of important results in this paper spring from descriptive proximally continuous maps.  A descriptive proximally continuous map is defined over descriptive \v{C}ech proximity spaces~\cite[\S 4.1]{DiConcilio2018MCSdescriptiveProximities} in which the description of a nonempty set is in the form of a feature vector derived from a probe function.  

For a \v{C}ech proximity space $X$ containing a nonempty subset $A\in 2^X$ (collection of subsets of $X$), a probe function $\Phi:2^X\to \mathbb{R}^n$ in which $\Phi(A)$ is a feature vector of real values that describe $A$.  Pivotal in this work is the descriptive intersection~\cite{DiConcilio2018MCSdescriptiveProximities} (denoted by $\dcap$) between nonvoid sets $A,B\in 2^X$, which is set of points $x$ in $A\cup B$ with the requirement that $\Phi(x)$ in the intersection of the pair of descriptions $\Phi(A),\Phi(B)$.  

Let $\near$ be a \v{C}ech proximity and $(2^X,{\near}_1), (2^X,{\near}_2)$ be a pair of proximity spaces, $A,B\in 2^X$.  A map $(2^X,{\near}_1)\to (2^X,{\near}_2)$ is proximally continuous, provided $A\ \near\ B$ implies $f(A)\ \near\ f(B)$
\cite[\S 4, p. 20]{Naimpally70}. The notion of descriptive proximal continuity is an easy step beyond the traditional notion of proximal continuity.  Let $\dnear$ denote a descriptive proximity relation and let $(2^X,{\dnear}_1), (2^X,{\dnear}_2)$ be a pair of descriptive proximity spaces.  A map $(2^X,{\near}_1)\to (2^X,{\near}_2)$ is a descriptive proximally continuous map~\cite[\S 1.20.1, p. 48]{Peters2013springer}, provided $A\ \dnear\ B$ implies $f(A)\ \dnear\ f(B)$.

\section{Preliminaries}
This section briefly introduces descriptive \v{C}ech proximity spaces and descriptively proximal continuity.  The simplest form of proximity relation (denoted by $\delta$) on a nonempty set was intoduced by E. \v{C}ech~\cite{Cech1966}.  A nonempty set $X$ equipped with the relation $\near$ is a \v{C}ech proximity space (denoted by $(X,\near$)), provided the following axioms are satisfied.\\
\vspace{3mm}

\noindent {\bf \v{C}ech Axioms}

\begin{description}
\item[({\bf P}.0)] All nonempty subsets in $X$ are far from the empty set, {\em i.e.}, $A\ \not{\near}\ \emptyset$ for all $A\subseteq X$.
\item[({\bf P}.1)] $A\ \near\ B \Rightarrow B\ \near\ A$.
\item[({\bf P}.2)] $A\ \cap\ B\neq \emptyset \Rightarrow A\ \near\ B$.
\item[({\bf P}.3)] $A\ \near\ \left(B\cup C\right) \Rightarrow A\ \near\ B$ or $A\ \near\ C$.
\end{description}

\subsection{Descriptive Proximity} 
 Given that a nonempty set $E$ has $k \geq 1$ features such as Fermi energy $E_{Fe}$, cardinality $E_{card}$, a description $\Phi(E)$ of $E$ is a feature vector, {\em i.e.}, $\Phi(E) = \left(E_{Fe},E_{card}\right)$. Nonempty sets $A,B$ with overlapping descriptions are descriptively proximal (denoted by $A\ \dnear\ B$). 
  The descriptive intersection of nonempty subsets in $A\cup B$ (denoted by $A\ \dcap\ B$) is defined by
\[
A\ \dcap\ B = \overbrace{\left\{x\in A\cup B: \Phi(x) \in \Phi(A)\ \cap\ \Phi(B)\right\}.}^{\mbox{\textcolor{blue}{\bf {\em i.e.}, $\boldsymbol{\mbox{Descriptions}\ \Phi(A)\ \&\ \Phi(B)\ \mbox{overlap}}$}}}
\] 

Let $2^X$ denote the collection of all subsets in a nonvoid set $X$. A nonempty set $X$ equipped with the relation $\dnear$ with non-void subsets $A,B,C\in 2^X$ is a descriptive proximity space, provided the following descriptive forms of the \v{C}ech axioms are satisfied.\\
\vspace{3mm}

\noindent {\bf Descriptive \v{C}ech Axioms}

\begin{description}
\item[({\bf dP}.0)] All nonempty subsets in $2^X$ are descriptively far from the empty set, {\em i.e.}, $A\ \not{\dnear}\ \emptyset$ for all $A\in 2^X$.
\item[({\bf dP}.1)] $A\ \dnear\ B \Rightarrow B\ \dnear\ A$.
\item[({\bf dP}.2)] $A\ \dcap\ B\neq \emptyset \Rightarrow A\ \dnear\ B$.
\item[({\bf dP}.3)] $A\ \dnear\ \left(B\cup C\right) \Rightarrow A\ \dnear\ B$ or $A\ \dnear\ C$.
\end{description}

\noindent The converse of Axiom ({\bf dp}.2) also holds.
\begin{lemma}\label{lemma:dP2converse}{\rm \cite{Peters2019vortexNerves}}
Let $X$ be equipped with the relation $\dnear$, $A,B\in 2^X$.   Then $A\ \dnear\ B$ implies $A\ \dcap\ B\neq \emptyset$.
\end{lemma}
\begin{proof}
Let $A,B\in 2^X$. By definition, $A\ \dnear\ B$ implies that there is at least one member $x\in A$ and $y\in B$ so that $\Phi(x) = \Phi(y)$, {\em i.e.}, $x$ and $y$ have the same description.  Then $x,y\in A\ \dcap\ B$.
Hence, $A\ \dcap\ B\neq \emptyset$, which is the converse of ({\bf dp}.2).    
\end{proof}

\begin{theorem}\label{theorem:dP2result}
Let $K$ be a cell complex and $\crb(K)$ be the  collection of planar ribbon complexes equipped with the proximity $\dnear$ in $K$ and  $\rb A,\rb B\in \crb(K)$.   Then $\rb A\ \dnear\ \rb B$ implies $\rb A\ \dcap\ \rb B\neq \emptyset$.
\end{theorem}
\begin{proof}
Immediate from Lemma~\ref{lemma:dP2converse}. 
\end{proof}

\begin{figure}[!ht]
\centering
\subfigure[Ribbon $\rb E$ with non-intersecting cycles]
 {\label{fig:rbE}
\begin{pspicture}
(-1.5,-0.5)(4.0,3.0)
\psframe[linewidth=0.75pt,linearc=0.25,cornersize=absolute,linecolor=blue](-1.25,-0.25)(3.25,3)
\psline*[linestyle=solid,linecolor=green!30]%
(0,0)(1,0.5)(2.0,0.0)(3.0,0.5)(3.0,1.5)(2.0,2.0)(1,1.5)(0,2)
(-1,1.5)(-1,0.5)(0,0)
\psline[linestyle=solid,linecolor=black]%
(0,0)(1,0.5)(2.0,0.0)(3.0,0.5)(3.0,1.5)(2.0,2.0)(1,1.5)(0,2)
(-1,1.5)(-1,0.5)(0,0)
\psdots[dotstyle=o,dotsize=2.2pt,linewidth=1.2pt,linecolor=black,fillcolor=gray!80]%
(0,0)(1,0.5)(2.0,0.0)(3.0,0.5)(3.0,1.5)(2.0,2.0)(1,1.5)(0,2)
(-1,1.5)(-1,0.5)(0,0)
\psline[linestyle=solid](1,1.5)(2,2)\psline[arrows=<->](-1,1.7)(-0.3,2.0)\psline[arrows=<->](0.2,2.05)(0.8,1.75)
\psline*[linestyle=solid,linecolor=white]%
(0,0.25)(1,0.75)(2.0,0.25)(2.5,0.5)(2.5,0.75)(2.0,1.35)(1,1.25)(0,1.5)
(-.55,1.25)(-.55,0.75)(0,0.25)
\psline[linestyle=solid,linecolor=black]%
(0,0.25)(1,0.75)(2.0,0.25)(2.5,0.5)(2.5,0.75)(2.0,1.35)(1,1.25)(0,1.5)
(-.55,1.25)(-.55,0.75)(0,0.25)
\psdots[dotstyle=o,dotsize=2.2pt,linewidth=1.2pt,linecolor=black,fillcolor=gray!80]%
(0,0.25)(1,0.75)(2.0,0.25)(2.5,0.5)(2.5,0.75)(2.0,1.35)(1,1.25)(0,1.5)
(-.55,1.25)(-.55,0.75)(0,0.25)
\rput(-1.0,2.75){\footnotesize $\boldsymbol{K}$}
\rput(0.0,1.75){\footnotesize $\boldsymbol{rb E}$}
\rput(2.8,1.85){\footnotesize $\boldsymbol{cyc A}$}
\rput(2.5,1.25){\footnotesize $\boldsymbol{cyc B}$}
\end{pspicture}}\hfil
\subfigure[Ribbon $\rb E'$ with intersecting cycles]
 {\label{fig:rbNrvE}
\begin{pspicture}
(-1.5,-0.5)(4.0,3.0)
\psframe[linewidth=0.75pt,linearc=0.25,cornersize=absolute,linecolor=blue](-1.35,-0.25)(3.25,3)
\psline*[linestyle=solid,linecolor=green!30]%
(0,0)(1,0.5)(2.0,0.0)(3.0,0.5)(3.0,1.5)(2.0,2.0)(1,1.5)(0,2)
(-1,1.5)(-1,0.5)(0,0)
\psline[linestyle=solid,linecolor=black]%
(0,0)(1,0.5)(2.0,0.0)(3.0,0.5)(3.0,1.5)(2.0,2.0)(1,1.5)(0,2)
(-1,1.5)(-1,0.5)(0,0)
\psdots[dotstyle=o,dotsize=2.2pt,linewidth=1.2pt,linecolor=black,fillcolor=gray!80]%
(0,0)(1,0.5)(2.0,0.0)(3.0,0.5)(3.0,1.5)(2.0,2.0)(1,1.5)(0,2)
(-1,1.5)(-1,0.5)(0,0)
\psline[linestyle=solid](1,1.5)(2,2)\psline[arrows=<->](-1,1.7)(-0.3,2.0)\psline[arrows=<->](0.2,2.05)(0.8,1.75)
\psline*[linestyle=solid,linecolor=white]%
(0,0.25)(1,0.5)(2.0,0.25)(2.5,0.5)(2.5,0.75)(2.0,2.0)(1,1.25)(0,1.5)
(-.55,1.25)(-1,1.5)(0,0.25)
\psline[linestyle=solid,linecolor=black]%
(0,0.25)(1,0.5)(2.0,0.25)(2.5,0.5)(2.5,0.75)(2.0,2.0)(1,1.25)(0,1.5)
(-.55,1.25)(-1,1.5)(0,0.25) 
\psdots[dotstyle=o,dotsize=2.2pt,linewidth=1.2pt,linecolor=black,fillcolor=gray!80]%
(0,0.25)(1,0.5)(2.0,0.25)(2.5,0.5)(2.5,0.75)(2.0,2.0)(1,1.25)(0,1.5)
(-.55,1.25)(-1,1.5)(0,0.25)
\psdots[dotstyle=o,dotsize=2.5pt,linewidth=1.2pt,linecolor=black,fillcolor=red!80]
(-1,1.5)(2.0,2.0)(1,0.5)
\rput(-1.0,2.75){\footnotesize $\boldsymbol{K'}$}
\rput(-1.1,1.6){\footnotesize $\boldsymbol{a}$}
\rput(2.0,2.15){\footnotesize $\boldsymbol{a'}$}
\rput(1,0.65){\footnotesize $\boldsymbol{a''}$}
\rput(0.0,1.655){\footnotesize $\boldsymbol{rb E'}$}
\rput(2.8,1.85){\footnotesize $\boldsymbol{cyc A'}$}
\rput(2.5,1.25){\footnotesize $\boldsymbol{cyc B'}$}
l\end{pspicture}}
\caption[]{Sample planar ribbon structures}
\label{fig:2Ribbons}
\end{figure}

\subsection{Planar Ribbons}\label{sec:ribbonStructures}
This section briefly looks at planar ribbon structures in planar CW spaces. 
Briefly, A nonvoid collection of cell complexes $K$ is a \emph{Closure finite Weak} (CW) space, provided $K$ is Hausdorff (every pair of distinct cells is contained in disjoint neighbourhoods~\cite[\S 5.1, p. 94]{Naimpally2013}) and the collection of cell complexes in $K$ satisfy the Alexandroff-Hopf-Whitehead~\cite[\S III, starting on page 124]{AlexandroffHopf1935Topologie},~\cite[pp. 315-317]{Whitehead1939homotopy}, ~\cite[\S 5, p. 223]{Whitehead1949BAMS-CWtopology} conditions, namely, containment (the closure of each cell complex is in $K$) and intersection (the nonempty intersection of cell complexes is in $K$). Each planar ribbon contains a pair of nested, usually non-concentric filled cycles. A ribbon cycle is a simple closed curve defined by a sequence of path-connected vertexes.  Cycle Vertices are path-connected, provided there is a sequence of edges between each pair of vertices in the cycle.  A ribbon cycle is filled, since the interior of each cycle is nonempty.   Planar ribbons spring naturally from the exclusion of the interior of the interior of the inner cycle in a pair of nested cycles.  

\begin{definition} Planar Ribbon\label{def:ribbon}{\rm \cite{Peters2020AMSBullRibbonComplexes}}.\\
Let $\cyc A, \cyc B$ be nesting filled cycles (with $\cyc B$ in the interior of $\cyc A$) defined on a finite, bounded, planar region in a CW space $K$.  A \emph{planar ribbon} $E$ (denoted by $\rb E$) is defined by
\[
\rb E = \overbrace{
\left\{\cl(\cyc A)\setminus \left\{\cl(\cyc B)\setminus \Int(\cyc B)\right\}: \bdy(\cl(\cyc B))\subset \cl(\rb E)\right\}.}^{\mbox{\textcolor{blue}{\bf $\bdy(\cl(\cyc B))$ defines the inner boundary of $\cl(\rb E)$.}}}\mbox{\textcolor{blue}{\Squaresteel}}
\]
\end{definition}

\begin{remark}
From Def.~\ref{def:ribbon}, the intersection of a pair of nested cycles $\cyc A,\cyc B$ in a ribbon $\rb E$ can be either empty ({\em e.g.}, $\rb E$ cycles in Fig.~\ref{fig:rbE} have no vertexes in common) or non-empty ({\em e.g.}, cycles $\cyc A',\cyc B'\in \rb E'$ in Fig.~\ref{fig:rbNrvE}, $\cyc A'\cap \cyc B' = \left\{a,a',a''\right\}$).   In addition, zero or more edges can be attached between ribbon cycles. 
\textcolor{blue}{\Squaresteel}
\end{remark}

In Example~\ref{ex:2ribbons}, the Betti number\footnote{an intuitive form of Betti number introduced by A.J. Zomorodian in~\cite[\S 4.3.2, p. 57]{Zomorodian2001BettiNumbers}.} $\beta_0(\cyc A)$ is a count of the number of vertexes in ribbon cycle $\cyc A$.

\begin{example}\label{ex:2ribbons}
A pair of sample planar ribbon structures is shown in Fig.~\ref{fig:2Ribbons}.  In Fig.~\ref{fig:rbE}, ribbon $\rb E$ is defined by a pair of non-intersecting cycles $\cyc A,\cyc B$.  Assume that complex $K$ is equipped with the descriptive proximity $\dnear$ with $\Phi(\rb E)$ equal to the number of vertices in its outer cycle $\cyc A$.  Similarly, let complex $K'$ in Fig.~\ref{fig:rbNrvE} be equipped with the descriptive proximity $\dnear$ with $\Phi(\rb E')$ equal to the number of vertices in its outer cycle $\cyc A'$.  In that case, $\rb E\ \dnear\ \rb E'$, since
\[ 
\Phi(\rb E) = \beta_0(\cyc A) = \Phi(\rb E') = \beta_0(\cyc A').
\]
Hence, from Lemma~\ref{lemma:dP2converse}, $\rb E\ \dnear\ \rb E'\Rightarrow \rb E\ \dcap\ \rb E'\neq \emptyset$.
\textcolor{blue}{\Squaresteel}
\end{example}

\begin{figure}[!ht]
\centering
\subfigure[Ribbon $\rb E''$ cycles with common vertex $g$]
 {\label{fig:rbE2}
\begin{pspicture}
(-1.5,-0.5)(4.0,3.0)
\psframe[linewidth=0.75pt,linearc=0.25,cornersize=absolute,linecolor=blue](-1.25,-0.25)(3.4,3)
\psline*[linestyle=solid,linecolor=green!30]%
(0,0)(1,0.5)(2.0,0.0)(3.0,0.5)(3.0,1.5)(2.0,2.0)(1,1.5)(0,2)
(-1,1.5)(-1,0.5)(0,0)
\psline[linestyle=solid,linecolor=black]%
(0,0)(1,0.5)(2.0,0.0)(3.0,0.5)(3.0,1.5)(2.0,2.0)(1,1.5)(0,2)
(-1,1.5)(-1,0.5)(0,0)
\psdots[dotstyle=o,dotsize=2.2pt,linewidth=1.2pt,linecolor=black,fillcolor=gray!80]%
(0,0)(1,0.5)(2.0,0.0)(3.0,0.5)(3.0,1.5)(2.0,2.0)(1,1.5)(0,2)
(-1,1.5)(-1,0.5)(0,0)
\psline[linestyle=solid](1,1.5)(2,2)\psline[arrows=<->](-1,1.7)(-0.3,2.0)\psline[arrows=<->](0.2,2.05)(0.8,1.75)
\psline*[linestyle=solid,linecolor=white]%
(0,0.25)(1,0.75)(2.0,0.25)(2.5,0.5)(3.0,0.5)(2.0,1.35)(1,1.25)(0,1.5)
(-.55,1.25)(-.55,0.75)(0,0.25)
\psline[linestyle=solid,linecolor=black]%
(0,0.25)(1,0.75)(2.0,0.25)(2.5,0.5)(3.0,0.5)(2.0,1.35)(1,1.25)(0,1.5)
(-.55,1.25)(-.55,0.75)(0,0.25)
\psdots[dotstyle=o,dotsize=2.2pt,linewidth=1.2pt,linecolor=black,fillcolor=gray!80]%
(0,0.25)(1,0.75)(2.0,0.25)(2.5,0.5)(3.0,0.5)(2.0,1.35)(1,1.25)(0,1.5)
(-.55,1.25)(-.55,0.75)(0,0.25)
\psdots[dotstyle=o,dotsize=2.5pt,linewidth=1.2pt,linecolor=black,fillcolor=red!80]
(3.0,0.5)
\rput(-1.0,2.78){\footnotesize $\boldsymbol{K''}$}
\rput(2.5,0.65){\footnotesize $\boldsymbol{b_1}$}
\rput(2.0,0.45){\footnotesize $\boldsymbol{b_2}$}
\rput(1,0.95){\footnotesize $\boldsymbol{b_3}$}
\rput(3.2,0.6){\footnotesize $\boldsymbol{g}$}
\rput(0.0,1.75){\footnotesize $\boldsymbol{rb E''}$}
\rput(2.8,1.85){\footnotesize $\boldsymbol{cyc A''}$}
\rput(2.5,1.25){\footnotesize $\boldsymbol{cyc B''}$}
\end{pspicture}}\hfil
\subfigure[Ribbon $\rb E'''$ cycles with 3 common vertexes $g_1,g_2$]{\label{fig:rbNrvE3}
\begin{pspicture}
(-1.5,-0.5)(4.0,3.0)
\psframe[linewidth=0.75pt,linearc=0.25,cornersize=absolute,linecolor=blue](-1.35,-0.25)(3.4,3)
\psline*[linestyle=solid,linecolor=green!30]%
(0,0)(1,0.5)(2.0,0.0)(3.0,0.5)(3.0,1.5)(2.0,2.0)(1,1.5)(0,2)
(-1,1.5)(-1,0.5)(0,0)
\psline[linestyle=solid,linecolor=black]%
(0,0)(1,0.5)(2.0,0.0)(3.0,0.5)(3.0,1.5)(2.0,2.0)(1,1.5)(0,2)
(-1,1.5)(-1,0.5)(0,0)
\psdots[dotstyle=o,dotsize=2.2pt,linewidth=1.2pt,linecolor=black,fillcolor=gray!80]%
(0,0)(1,0.5)(2.0,0.0)(3.0,0.5)(3.0,1.5)(2.0,2.0)(1,1.5)(0,2)
(-1,1.5)(-1,0.5)(0,0)
\psline[linestyle=solid](1,1.5)(2,2)\psline[arrows=<->](-1,1.7)(-0.3,2.0)\psline[arrows=<->](0.2,2.05)(0.8,1.75)
\psline*[linestyle=solid,linecolor=white]%
(0,0.25)(1,0.75)(2.0,0.25)(3.0,0.5)(2.5,0.75)(2.0,2.0)(1,1.25)(0,1.5)
(-.55,1.25)(-1,1.5)(0,0.25)
\psline[linestyle=solid,linecolor=black]%
(0,0.25)(1,0.75)(2.0,0.25)(3.0,0.5)(2.5,0.75)(2.0,2.0)(1,1.25)(0,1.5)
(-.55,1.25)(-1,1.5)(0,0.25) 
\psdots[dotstyle=o,dotsize=2.2pt,linewidth=1.2pt,linecolor=black,fillcolor=gray!80]%
(0,0.25)(1,0.75)(2.0,0.25)(3.0,0.5)(2.5,0.75)(2.0,2.0)(1,1.25)(0,1.5)
(-.55,1.25)(-1,1.5)(0,0.25)
\psdots[dotstyle=o,dotsize=2.5pt,linewidth=1.2pt,linecolor=black,fillcolor=red!80]
(-1,1.5)(3.0,0.5)(2.0,2.0)
\rput(-1.0,2.75){\footnotesize $\boldsymbol{K'''}$}
\rput(0,0.48){\footnotesize $\boldsymbol{b_1}$}
\rput(2.0,0.45){\footnotesize $\boldsymbol{b_2}$}
\rput(1,0.95){\footnotesize $\boldsymbol{b_3}$}
\rput(-1.1,1.65){\footnotesize $\boldsymbol{g_1}$}
\rput(3.15,0.6){\footnotesize $\boldsymbol{g_2}$}
\rput(2.0,2.15){\footnotesize $\boldsymbol{g_3}$}
\rput(0.0,1.655){\footnotesize $\boldsymbol{rb E'''}$}
\rput(2.8,1.85){\footnotesize $\boldsymbol{cyc A'''}$}
\rput(2.5,1.25){\footnotesize $\boldsymbol{cyc B'''}$}
\end{pspicture}}
\caption[]{Intersecting ribbon cycles}
\label{fig:3Ribbons}
\end{figure}

\subsection{Free Abelian Group Representation of a Ribbon with Intersecting Cycles}

Recall that a finite group $G$ is cyclic, provided every element $g\in G$ is an integer multiple of an element $a\in G$ (called the generator of $G$ and denoted by $\langle a\rangle$), {\em i.e.}, $g = a+\cdots + a = ka, k > 0$~\cite{Rotman1965groups}.  A free fg Abelian group $G$ representation of a ribbon $\rb E$ is the natural outcome of intersecting ribbon cycles (denoted by $G(+,\left\{\langle a\rangle\right\})$). For a pair of end points $a,a'$ on an edge $\arc{a,a'}$ in ribbon $\rb E$, $a+a'$ reads 'move from $a$ to $a'$'.  For example, in ribbon $\rb E'$ in Fig.~\ref{fig:rbNrvE} has vertex $a\in \cyc A'\cap \cyc B'$ defines a cyclic group on each of the cycles in $\rb E'$, {\em i.e.}, each vertex $a'\in \cyc A'$ or vertex $b'\in \cyc B'$ is reached by a multiple of moves from vertex $a$ to vertex $a'$ or to vertex $b'$.  The assumption here that $a' = a + a +\cdots + a = ka$, where $ka$ reads 'sequence of $k$ moves in the path between vertex $a$ and vertex $a'$. Recall that a group $G$ is a free finitely-generated (fg) group $G$ with generators $a_1,\dots,a_k$, provided each element $g\in G$ can be written as a linear combination of its generators~\cite{Goblin2010CUPhomology}. 

\begin{example}
A ribbon $\rb E''$ with cycles $\cyc A'',\cyc B''$ intersection at a point $g$ is shown in Fig.~\ref{fig:rbE2}.  Each of the cycle vertexes in this ribbon can be written as a linear combination of $g$.  For example, we have
\[
b_3 = \overbrace{g + b_1 + b_2 + b_3 \mapsto 3g.}^{\mbox{\textcolor{blue}{\bf 3 moves from $g$ to reach $b_3$}}}
\]
The cycles in ribbon $\rb E'''$ in Fig.~\ref{fig:rbNrvE3} intersect in $g_1,g_2,g_3$, which serve as generators of cyclic groups on the ribbon.  
Similarly,
\[
b_1 = \overbrace{g_2 + b_2 + b_3 + b_1 \mapsto 3g_2 + 0g_1.}^{\mbox{\textcolor{blue}{\bf 3 moves from $g_2$ + 0 moves from $g_1$\ \&\ $g_3$ to reach $b_1$}}}
\]
By choosing path containing a minimal number of moves between a generator and a target ribbon vertex, we obtain a unique linear combination of the generators for each ribbon vertex.
As a result, ribbon $\rb E'''$ is represented by a free fg Abelian group with generators $\langle g_1\rangle, \langle g_2\rangle, \langle g_3\rangle$.
\textcolor{blue}{\Squaresteel}
\end{example}

The Betti number $\beta_{\alpha}(G)$ is a count of the number of generators in a free fg Abelian group $G$ (rank of $G$)~\cite[\S 1.4, p. 24]{Munkres1984}. For example, vertexes $a$ in the nonempty intersection of cycles in a ribbon define a free fg Abelian group $G$ with respect to a set of generators $\left\{\langle a\rangle\right\})$,
 since there is at least one path between the generator vertex and any other ribbon vertex.  In effect, each vertex common to the cycles of ribbon $\rb E$ serves as a generator of the free fg Abelian group representing $\rb E$. 

\begin{lemma}\label{lem:bridgeGenerators}
Let $\rb E$ in a CW complex $K$ have intersecting cycles with $n$ vertexes in the intersection.  Then
$\beta_{\alpha}(\rb E) = n$.
\end{lemma}
\begin{proof}
Let $g$ be a vertex in the intersection of the cycles in ribbon $\rb E$.  Each cycle vertex $a\in \rb E$ can be written as a sequence of  $m$ moves from $g$ to $a$, {\em i.e.}, $a = mg$, since vertex $g$ in common to both cycles in $\rb E$. In effect, $g$ is the generator of a cyclic group on $\rb E$. For a ribbon with $n > 0$ points of intersection of its cycles, there are $n$ cyclic groups on the ribbon.  Hence, ribbon $\rb E$ has free fg Abelian group representation derived from the set of points in the nonempty intersection of its cycles.  This gives the desired result, namely, $\beta_{\alpha}(\rb E) = n$.
\end{proof}

\begin{example}
In Fig.~\ref{fig:rbNrvE3}, ribbon $\rb E'''$ in a CW complex $K'''$ is defined by a pair of intersecting cycles $\cyc A''',\cyc B'''$.  Assume that complex $K'''$ is equipped with the descriptive proximity $\dnear$ with $\Phi(\rb E''')$ equal to $\beta_{\alpha}(\rb E''')$  (the number of generators in the free fg Abelian group representation of $\rb E''$) = 3 (from Lemma~\ref{lem:bridgeGenerators}).  Similarly, let complex $K'$ in Fig.~\ref{fig:rbNrvE} be equipped with the descriptive proximity $\dnear$ with $\Phi(\rb E')$ equal to $\beta_{\alpha}(\rb E')$ = 3.  Then we have $\rb E'''\ \dnear\ \rb E'$, since
\[ 
\Phi(\rb E''') = \beta_{\alpha}(\rb E''') = \Phi(\rb E') = \beta_{\alpha}(\rb E').
\]
Hence, from Lemma~\ref{lemma:dP2converse}, $\rb E\ \dnear\ \rb E'\Rightarrow \rb E\ \dcap\ \rb E'\neq \emptyset$.
However, in Fig.~\ref{fig:rbE2}, $\rb E''\ \not{\dnear}\ \rb E'''$, since
\[ 
\Phi(\rb E'') = \beta_{\alpha}(\rb E'') \neq \Phi(\rb E''') = \beta_{\alpha}(\rb E''').
\ \mbox{\textcolor{blue}{\Squaresteel}}
\]
\end{example}

\begin{figure}[!ht]
\centering
\subfigure[Ribbon $\rb E_5$ with 1 bridge edge between cycles]
 {\label{fig:1BridgeEdge}
\begin{pspicture}
(-1.5,-0.5)(4.0,3.0)
\psframe[linewidth=0.75pt,linearc=0.25,cornersize=absolute,linecolor=blue](-1.25,-0.25)(3.4,3)
\psline*[linestyle=solid,linecolor=green!30]%
(0,0)(1,0.5)(2.0,0.0)(3.0,0.5)(3.0,1.5)(2.0,2.0)(1,1.5)(0,2)
(-1,1.5)(-1,0.5)(0,0)
\psline[linestyle=solid,linecolor=black]%
(0,0)(1,0.5)(2.0,0.0)(3.0,0.5)(3.0,1.5)(2.0,2.0)(1,1.5)(0,2)
(-1,1.5)(-1,0.5)(0,0)
\psline[linestyle=solid,linecolor=black]
(1,1.5)(1,1.25)
\psdots[dotstyle=o,dotsize=2.2pt,linewidth=1.2pt,linecolor=black,fillcolor=gray!80]%
(0,0)(1,0.5)(2.0,0.0)(3.0,0.5)(3.0,1.5)(2.0,2.0)(1,1.5)(0,2)
(-1,1.5)(-1,0.5)(0,0)
\psline[linestyle=solid](1,1.5)(2,2)\psline[arrows=<->](-1,1.7)(-0.3,2.0)\psline[arrows=<->](0.2,2.05)(0.8,1.75)
\psline*[linestyle=solid,linecolor=white]%
(0,0.25)(1,0.75)(2.0,0.25)(2.5,0.5)(2.5,0.75)(2.0,1.35)(1,1.25)(0,1.5)
(-.55,1.25)(-.55,0.75)(0,0.25)
\psline[linestyle=solid,linecolor=black]%
(0,0.25)(1,0.75)(2.0,0.25)(2.5,0.5)(2.5,0.75)(2.0,1.35)(1,1.25)(0,1.5)
(-.55,1.25)(-.55,0.75)(0,0.25)
\psdots[dotstyle=o,dotsize=2.2pt,linewidth=1.2pt,linecolor=black,fillcolor=gray!80]%
(0,0.25)(1,0.75)(2.0,0.25)(2.5,0.5)(2.5,0.75)(2.0,1.35)(1,1.25)(0,1.5)
(-.55,1.25)(-.55,0.75)(0,0.25)
\psdots[dotstyle=o,dotsize=2.5pt,linewidth=1.2pt,linecolor=black,fillcolor=red!80]
(1,1.5)(1,1.25)
\rput(-1.0,2.75){\footnotesize $\boldsymbol{K_5}$}
\rput(1,1.72){\footnotesize $\boldsymbol{\arc{e_{_0}}}$}\rput(1,1.05){\footnotesize $\boldsymbol{\arc{e_{_1}}}$}
\rput(0.0,1.75){\footnotesize $\boldsymbol{rb E_5}$}
\rput(2.8,1.85){\footnotesize $\boldsymbol{cyc A_5}$}
\rput(2.5,1.25){\footnotesize $\boldsymbol{cyc B_5}$}
\end{pspicture}}\hfil
\subfigure[Ribbon $\rb E_6$ with 2 bridge edges between cycles]
 {\label{fig:2BridgeEdges}
\begin{pspicture}
(-1.5,-0.5)(4.0,3.0)
\psframe[linewidth=0.75pt,linearc=0.25,cornersize=absolute,linecolor=blue](-1.35,-0.25)(3.4,3)
\psline*[linestyle=solid,linecolor=green!30]%
(0,0)(1,0.5)(2.0,0.0)(3.0,0.5)(3.0,1.5)(2.0,2.0)(1,1.5)(0,2)
(-1,1.5)(-1,0.5)(0,0)
\psline[linestyle=solid,linecolor=black]%
(0,0)(1,0.5)(2.0,0.0)(3.0,0.5)(3.0,1.5)(2.0,2.0)(1,1.5)(0,2)
(-1,1.5)(-1,0.5)(0,0)
\psline[linestyle=solid,linecolor=black]
(1,1.5)(1,1.25)
\psline[linestyle=solid,linecolor=black]
(2.0,1.4)(2.0,2.0)
\psdots[dotstyle=o,dotsize=2.2pt,linewidth=1.2pt,linecolor=black,fillcolor=gray!80]%
(0,0)(1,0.5)(2.0,0.0)(3.0,0.5)(3.0,1.5)(2.0,2.0)(1,1.5)(0,2)
(-1,1.5)(-1,0.5)(0,0)
\psline[linestyle=solid](1,1.5)(2,2)\psline[arrows=<->](-1,1.7)(-0.3,2.0)\psline[arrows=<->](0.2,2.05)(0.8,1.75)
\psline*[linestyle=solid,linecolor=white]%
(0,0.25)(1,0.75)(2.0,0.25)(2.5,0.5)(2.5,1.25)(2.0,1.55)(1,1.25)(0,1.5)
(-.55,1.25)(-.55,0.75)(0,0.25)
\psline[linestyle=solid,linecolor=black]%
(0,0.25)(1,0.75)(2.0,0.25)(2.5,0.5)(2.5,1.25)(2.0,1.55)(1,1.25)(0,1.5)
(-.55,1.25)(-.55,0.75)(0,0.25) 
\psdots[dotstyle=o,dotsize=2.2pt,linewidth=1.2pt,linecolor=black,fillcolor=gray!80]%
(0,0.25)(1,0.75)(2.0,0.25)(2.5,0.5)(2.5,1.25)(2.0,1.55)(1,1.25)(0,1.5)
(-.55,1.25)(-.55,0.75)(0,0.25)
\psdots[dotstyle=o,dotsize=2.5pt,linewidth=1.2pt,linecolor=black,fillcolor=red!80]
(1,1.5)(1,1.25)(2,2.0)(2.0,1.55)
\rput(-1.0,2.75){\footnotesize $\boldsymbol{K_6}$}
\rput(1,1.78){\footnotesize $\boldsymbol{\arc{g_{_0}}}$}\rput(1,1.05){\footnotesize $\boldsymbol{\arc{g_{_1}}}$}\rput(2,2.25){\footnotesize $\boldsymbol{\arc{g_{_2}}}$}\rput(2.0,1.3){\footnotesize $\boldsymbol{\arc{g_{_3}}}$}
\rput(-1.1,1.65){\footnotesize $\boldsymbol{a_1}$}
\rput(3.15,0.6){\footnotesize $\boldsymbol{a_2}$}
\rput(0.0,1.655){\footnotesize $\boldsymbol{rb E_6}$}
\rput(2.9,1.85){\footnotesize $\boldsymbol{cyc A_6}$}
\rput(2.52,0.8){\footnotesize $\boldsymbol{cyc B_6}$}
\end{pspicture}}
\caption[]{Ribbon bridge edges}
\label{fig:RibbonBridges}
\end{figure}

\subsection{Ribbon Bridge Edges}
Another source of generators of free fg Abelian group representations of ribbons are the endpoints of bridge edges attached between ribbon cycles. A ribbon \emph{bridge edge} is an edge attached between ribbon cycle vertexes.

\begin{example}
In Fig.~\ref{fig:1BridgeEdge}, vertex $\arc{e_{_0}}$ on cycle $\cyc A_5$ and vertex $\arc{e_{_1}}$ on cycle $\cyc B_5$ are the endpoints of bridge edge $\arc{e_{_0},e_{_1}}$ attached between the cycles on ribbon $\rb E_5$.  Similarly, edges $\arc{g_{_0},g_{_1}},\arc{g_{_2},g_{_3}}$ are a pair of bridge segments attached between cycles $\cyc A_6, \cyc A_6$ on ribbon $\rb E_6$ in Fig.~\ref{fig:2BridgeEdges}.
\textcolor{blue}{\Squaresteel}
\end{example}

In a ribbon that contains a bridge edge vertex $\arc{g}$, every ribbon vertex $v$ can be written a linear combination written of the vertexes in the path that starts with $\arc{g}$ and ends with $v$.  In effect, every ribbon bridge segment vertex $\arc{e}$ is a generator 
$\langle\arc{e}\rangle$ of a cyclic group.  Hence, a ribbon $\rb E$ with a single bridge edge like the one in Fig.~\ref{fig:1BridgeEdge} can be represented by a free fg Abelian group with two generators, {\em i.e.}, $\beta_{\alpha}(\rb E) = 2$.  In general, we have the following result.

\begin{lemma}\label{lem:bridgeSegmentVertexGenerators}
Let $\rb E$ in a CW complex $K$ have $k$ bridge segments.  Then
$\beta_{\alpha}(\rb E) = 2k$.
\end{lemma}

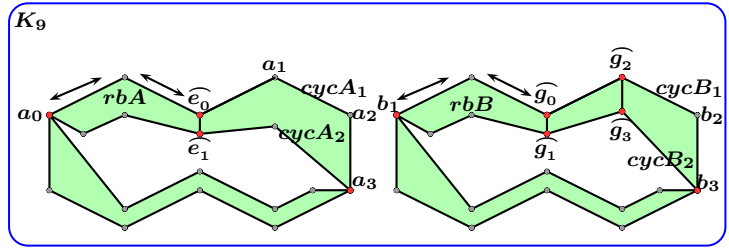
\begin{figure}[!ht]
\centering
\begin{pspicture}
(-1.8,-0.5)(3.0,3.0)
\psframe[linewidth=0.75pt,linearc=0.25,cornersize=absolute,linecolor=blue](-1.55,-0.25)(8,3.0)
\psline*[linestyle=solid,linecolor=green!30]%
(0,0)(1,0.5)(2.0,0.0)(3.0,0.5)(3.0,1.5)(2.0,2.0)(1,1.5)(0,2)
(-1,1.5)(-1,0.5)(0,0)
\psline[linestyle=solid,linecolor=black]%
(0,0)(1,0.5)(2.0,0.0)(3.0,0.5)(3.0,1.5)(2.0,2.0)(1,1.5)(0,2)
(-1,1.5)(-1,0.5)(0,0)
\psline[linestyle=solid,linecolor=black]
(1,1.5)(1,1.25)
\psdots[dotstyle=o,dotsize=2.2pt,linewidth=1.2pt,linecolor=black,fillcolor=gray!80]%
(0,0)(1,0.5)(2.0,0.0)(3.0,0.5)(3.0,1.5)(2.0,2.0)(1,1.5)(0,2)
(-1,1.5)(-1,0.5)(0,0)
\psline[linestyle=solid](1,1.5)(2,2)\psline[arrows=<->](-1,1.7)(-0.3,2.0)\psline[arrows=<->](0.2,2.05)(0.8,1.75)
\psline*[linestyle=solid,linecolor=white]%
(0,0.25)(1,0.75)(2.0,0.25)(2.5,0.5)(3.0,0.5)(2.0,1.35)(1,1.25)(0,1.5)
(-.55,1.25)(-1,1.5)(0,0.25)
\psline[linestyle=solid,linecolor=black]%
(0,0.25)(1,0.75)(2.0,0.25)(2.5,0.5)(3.0,0.5)(2.0,1.35)(1,1.25)(0,1.5)
(-.55,1.25)(-1,1.5)(0,0.25)
\psdots[dotstyle=o,dotsize=2.2pt,linewidth=1.2pt,linecolor=black,fillcolor=gray!80]%
(0,0.25)(1,0.75)(2.0,0.25)(2.5,0.5)(3.0,0.5)(2.0,1.35)(1,1.25)(0,1.5)
(-.55,1.25)(-1,1.5)(0,0.25)
\psdots[dotstyle=o,dotsize=2.5pt,linewidth=1.2pt,linecolor=black,fillcolor=red!80]
(1,1.5)(1,1.25)(3.0,0.5)(-1,1.5)
\rput(-1.25,2.75){\footnotesize $\boldsymbol{K_9}$}
\rput(1,1.72){\footnotesize $\boldsymbol{\arc{e_{_0}}}$}\rput(1,1.05){\footnotesize $\boldsymbol{\arc{e_{_1}}}$}
\rput(-1.25,1.5){\footnotesize $\boldsymbol{a_0}$}
\rput(2.0,2.15){\footnotesize $\boldsymbol{a_1}$}
\rput(3.2,1.5){\footnotesize $\boldsymbol{a_2}$}
\rput(3.2,0.6){\footnotesize $\boldsymbol{a_3}$}
\rput(0.0,1.75){\footnotesize $\boldsymbol{rb A}$}
\rput(2.8,1.85){\footnotesize $\boldsymbol{cyc A_1}$}
\rput(2.5,1.25){\footnotesize $\boldsymbol{cyc A_2}$}
\end{pspicture}
\begin{pspicture}
(-1.5,-0.5)(4.0,4.0)
\psline*[linestyle=solid,linecolor=green!30]%
(0,0)(1,0.5)(2.0,0.0)(3.0,0.5)(3.0,1.5)(2.0,2.0)(1,1.5)(0,2)
(-1,1.5)(-1,0.5)(0,0)
\psline[linestyle=solid,linecolor=black]%
(0,0)(1,0.5)(2.0,0.0)(3.0,0.5)(3.0,1.5)(2.0,2.0)(1,1.5)(0,2)
(-1,1.5)(-1,0.5)(0,0)
\psline[linestyle=solid,linecolor=black]
(1,1.5)(1,1.25)
\psline[linestyle=solid,linecolor=black]
(2.0,1.4)(2.0,2.0)
\psdots[dotstyle=o,dotsize=2.2pt,linewidth=1.2pt,linecolor=black,fillcolor=gray!80]%
(0,0)(1,0.5)(2.0,0.0)(3.0,0.5)(3.0,1.5)(2.0,2.0)(1,1.5)(0,2)
(-1,1.5)(-1,0.5)(0,0)
\psline[linestyle=solid](1,1.5)(2,2)\psline[arrows=<->](-1,1.7)(-0.3,2.0)\psline[arrows=<->](0.2,2.05)(0.8,1.75)
\psline*[linestyle=solid,linecolor=white]%
(0,0.25)(1,0.75)(2.0,0.25)(2.5,0.5)(3.0,0.5)(2.0,1.55)(1,1.25)(0,1.5)
(-.55,1.25)(-1,1.5)(0,0.25)
\psline[linestyle=solid,linecolor=black]%
(0,0.25)(1,0.75)(2.0,0.25)(2.5,0.5)(3.0,0.5)(2.0,1.55)(1,1.25)(0,1.5)
(-.55,1.25)(-1,1.5)(0,0.25) 
\psdots[dotstyle=o,dotsize=2.2pt,linewidth=1.2pt,linecolor=black,fillcolor=gray!80]%
(0,0.25)(1,0.75)(2.0,0.25)(2.5,0.5)(3.0,0.5)(2.0,1.55)(1,1.25)(0,1.5)
(-.55,1.25)(-1,1.5)(0,0.25)
\psdots[dotstyle=o,dotsize=2.5pt,linewidth=1.2pt,linecolor=black,fillcolor=red!80]
(1,1.5)(1,1.25)(2,2.0)(2.0,1.55)(3.0,0.5)(-1,1.5)
\rput(1,1.78){\footnotesize $\boldsymbol{\arc{g_{_0}}}$}\rput(1,1.05){\footnotesize $\boldsymbol{\arc{g_{_1}}}$}\rput(2,2.25){\footnotesize $\boldsymbol{\arc{g_{_2}}}$}\rput(2.0,1.3){\footnotesize $\boldsymbol{\arc{g_{_3}}}$}
\rput(-1.1,1.65){\footnotesize $\boldsymbol{b_1}$}
\rput(3.2,1.5){\footnotesize $\boldsymbol{b_2}$}
\rput(3.15,0.6){\footnotesize $\boldsymbol{b_3}$}
\rput(0.0,1.655){\footnotesize $\boldsymbol{rb B}$}
\rput(2.9,1.85){\footnotesize $\boldsymbol{cyc B_1}$}
\rput(2.52,0.9){\footnotesize $\boldsymbol{cyc B_2}$}
\end{pspicture}
\caption[]{Ribbon complex}
\label{fig:RibbonBridgesAndIntersections}
\end{figure}

It is also possible for a ribbon to have a combination of bridge segments attached between its cycles, which intersect at one of more points.

\begin{theorem}\label{thm:bridgesAndIntersectingCycles}
Let $\rb E$ be a ribbon with intersecting cycles having $n$ vertexes in the intersection and with $k$ bridge edges.  Then
\[
\beta_{\alpha}(\rb E) = 2k + n.
\]  
\end{theorem}
\begin{proof}
Immediate from Lemma~\ref{lem:bridgeSegmentVertexGenerators} and Lemma~\ref{lem:bridgeGenerators}.
\end{proof}

\begin{example}
A sample ribbon complex containing ribbons $\rb A, \rb B$ in a CW space $K_9$ is shown in Fig.~\ref{fig:RibbonBridgesAndIntersections}. $\rb A$ has one bridge edge $\arc{e_0,e_1}$ with endpoints $e_0,e_1$ and with vertexes $a_0,a_3$ in the intersection of cycles $\cyc A_1,\cyc A_2$ and $\rb B$ has two bridge edges $\arc{g_0,g_1}, \arc{g_2,g_3}$ with  vertexes $b_1,b_3$ in the intersection of cycles $\cyc B_1,\cyc B_2$.  From Theorem~\ref{thm:bridgesAndIntersectingCycles}, we have $\beta_{\alpha}(\rb A) = 4$ and $\beta_{\alpha}(\rb B) = 6$.
\textcolor{blue}{\Squaresteel}
\end{example}

\subsection{Proximally Continuous}

Let $(X,\delta_1)$ and $(Y,\delta_2)$ be two \v{C}ech proximity spaces. Then a map $f: (X,\delta_1)\to (Y,\delta_2)$ is \emph{proximal continuous}, provided  $A\ \delta_1\ B$ implies $f(A)\ \delta_2\ f(B)$, \emph{i.e.}, $f(A)\ \delta_2\ f(B)$, provided $f(A)\ \cap\ f(B)\neq \emptyset$ for $A, B\in 2^X$~\cite[\S 1.4]{Naimpally70}.  In general, a proximal continuous function preserves the nearness of pairs of sets~\cite[\S 1.7,p. 16]{Naimpally2013}. Further, $f$ is a \emph{proximal isomorphism}, provided $f$ is proximal continuous with a proximal continuous inverse $f^{-1}$.\\

\begin{definition}
	Let $(X,\delta)$ be a \v{C}ech proximity space and $f: (X, \delta) \to (X, \delta)$ a proximal continuous map. A set $A \in 2^X$ is said to be invariant with respect to $f$, provided $f(A)\subseteq A$. 
\end{definition}

Notice that if $A$ is an invariant set with respect to $f$, then $f^n(A)\subseteq A$ for all positive integer $n$.

\begin{theorem}
	Let $(X,\delta)$ be a \v{C}ech proximity space and  $f: (X, \delta) \to (X, \delta)$ a proximal continuous map. If $\{ A_i\}_{i \in I} \subseteq 2^X$ is a collection of invariant sets with respect to  $f$, then
	\begin{compactenum}
		\item[i)] 	$\cup_{i\in I} A_i$ is  an invariant set with respect to $f$, and 
		\item[ii)] $\cap_{i\in I} A_i$ is  an invariant set with respect to $f$. 
	\end{compactenum}
\end{theorem}
\begin{proof}
From our assumption, we have $f(A_i)\subseteq A_i$ for all $i\in I$ so that
\begin{compactenum}
 	\item[i)] $f(\cup_{i\in I} A_i)= \cup_{i\in I} f(A_i) \subseteq \cup_{i\in I} A_i$, and 
 	\item[ii)] $f(\cap_{i\in I} A_i) \subseteq \cap_{i\in I} f(A_i) \subseteq \cap_{i\in I} A_i$.
	\end{compactenum}
\end{proof}

\noindent Theorem~\ref{thm:LodatoProximity} holds, provided $\delta$ is a Lodato proximity~\cite{Lodato1962}.

\begin{theorem}\label{thm:LodatoProximity}
	Let $(X,\delta)$ be a \v{C}ech proximity space and  $f: (X, \delta) \to (X, \delta)$ a proximal continuous map. If $A \in 2^X$ is invariant with respect to $f$ then $\cl A$ is also invariant with respect to $f$.
\end{theorem}

\begin{definition}
	Let $(X,\delta)$ be a  \v{C}ech proximity space, $A\in 2^X$, and  $f : (X,\delta)\to (X,\delta)$ a \emph{proximal continuous map}.
\begin{compactenum}[(i)]
\item $A$ is  a fixed subset of $f$, provided  $f(A)=A$.
\item $A$ is an \emph{eventual fixed subset} of $f$, provided $A$ is not a fixed subset while $f^n(A)$ is a temporally constrained fixed set for some $n$, {\em i.e.}, $f^n(A)$ occurs at the end of a period of time as a result of some event ({\em cf.}~\cite[\S 8.1,p. 240]{AdamsFranzosa:2007}).    	 
\item $A$ is an \emph{almost fixed subset} of $f$, provided  $f(A)=A$ or $A \ \delta \ f(A)$~\cite{Boxer:2019, BoxKarLop:2016}. \\
\end{compactenum}
\end{definition}

\section{Descriptive Proximally Continuous Maps}
Let $(X,\delta_{\Phi_1})$ and $(Y,\delta_{\Phi_2})$  be  descriptive proximity spaces with  probe functions $\Phi_1: X \to \mathbb{R}^n$, $\Phi_2: Y \to \mathbb{R}^n$, and  $A,B \in 2^X$. Then a map $f : (X,\delta_{\Phi_1})\to (Y,\delta_{\Phi_2})$ is said to be \emph{descriptive proximally continuous}, provided $A\ \delta_{\Phi_1}\ B$ implies $f(A)\ \delta_{\Phi_2}\ f(B)$, \emph{i.e.}, $f(A)\ \delta_{\Phi_2} \ f(B)$, provided $f(A)\ \dcap \ f(B)\neq \emptyset$. Further $f$ is a \emph{descriptive proximal  isomorphism}, provided $f$ and its inverse $f^{-1}$  are  descriptively proximally continuous. \\



\begin{definition}
	Let $(X,\delta_{\Phi})$ be a descriptive  \v{C}ech proximity space and $f: (X, \delta_{\Phi}) \to (X, \delta_{\Phi})$ a descriptive proximally continuous map. A set $A \in 2^X$ is said to be descriptively invariant with respect to $f$, provided $\Phi(f(A))\subseteq \Phi(A)$. 
\end{definition}

Notice that if $A$ is a descriptively invariant set with respect to $f$, then $\Phi(f^n(A))\subseteq \Phi(A)$ for all positive integer $n$.

\begin{theorem}
	Let $(X, \delta_{\Phi})$ be a descriptive \v{C}ech proximity space and  $f: (X, \delta_{\Phi}) \to (X, \delta_{\Phi})$ a proximal descriptive continuous map. If $\{ A_i\}_{i \in I} \subseteq 2^X$ is a collection of descriptively invariant sets with respect to $f$, then
	\begin{compactenum}
		\item[i)] 	$\cup_{i\in I} A_i$ is   descriptively invariant with respect to  $f$, and 
		\item[ii)] $\cap_{i\in I} A_i$ is   descriptively invariant  with respect to  $f$. 
	\end{compactenum}
\end{theorem}
\begin{proof}
	From our assumption, we have $\Phi(f(A_i))\subseteq \Phi(A_i)$ for all $i\in I$ so that
	
	\begin{compactenum}
		\item[i)] \begin{align*}
		f(\cup_{i\in I} A_i) &= \cup_{i\in I} f(A_i) \\
		\Phi(f(\cup_{i\in I} A_i)) & = \Phi(\cup_{i\in I} f(A_i)) \\
		&=\cup_{i\in I} \Phi (f(A_i)) \\
		& \subseteq \cup_{i\in I} \Phi(A_i)
		\end{align*}
		and 
		
		\item[ii)] \begin{align*}
		f(\cap_{i\in I} A_i) &\subseteq  \cap_{i\in I} f(A_i) \\
		\Phi(f(\cap_{i\in I} A_i)) &\subseteq \Phi(\cap_{i\in I} f(A_i)) \\
		&\subseteq \cap_{i\in I} \Phi (f(A_i)) \\
		& \subseteq \cap_{i\in I} \Phi(A_i)
		\end{align*}
	\end{compactenum}
\end{proof}

\begin{theorem}
	Let $(X, \delta_{\Phi})$ be a descriptive \v{C}ech proximity space and  $f: (X, \delta_{\Phi}) \to (X, \delta_{\Phi})$ a descriptive proximally continuous map. If $A \in 2^X$ is  descriptively invariant with respect to $f$ then $\cl_{\delta_{\Phi}} A$ is also  descriptively invariant with respect to with respect to $f$.
\end{theorem}
\begin{proof}
The descriptive closure of a subset $A$ of $X$ is defined in~\cite[\S 1.21.2]{Peters2013springer} as follows: 	
\[ \cl_{\Phi} A=\{ x\in X \ | \ x \ \delta_{\Phi} \ A \}.\]
Take  an element $x$ in  $ \cl_{\Phi} A$ so that  $x\ \delta_{\Phi} \ A$ and  $\Phi(x) \in \Phi(A)$ by Lemma~\ref{lemma:dP2converse}.     Since $f$ is a descripitive proximally continuous $f(x) \ \delta_{\Phi} \ f(A)$ and $\Phi(f(x))\in \Phi(f(A))$ by Lemma~\ref{lemma:dP2converse}. We also have $\Phi(f(x))\in \Phi(A)$ since $A$ is an invariant set with respect to $f$. Therefore  $f(x)\ \delta_{\Phi} \ A$  and   $f(x) \in \cl_{\Phi} A$. Since this holds for all $x \in \cl_{\Phi}A$, we have $f(\cl_{\Phi} A) \subseteq \cl_{\Phi} A$ so that         $\Phi(f(\cl_{\Phi} A)) \subseteq \Phi(\cl_{\Phi} A)$.
\end{proof}

\begin{definition} \label{def:desfixed}
	Let $(X,\delta_{\Phi})$ be a descriptive proximity space with a probe function $\Phi: X \to \mathbb{R}^n$, $A\in 2^X$, and  $f : (X,\delta_{\Phi})\to (X,\delta_{\Phi})$ a descriptive proximally  continuous map.
	\begin{compactenum}[(i)]
\item $A$  is  a descriptive fixed subset of $f$, provided  $\Phi(f(A))=\Phi(A)$. 
\item $A$ and $f(A)$ are amiable fixed sets, provided $f(A) \ \dcap \ A \neq \emptyset$. 
\item $A$   is an eventual descriptive fixed subset of $f$, provided $A$ is not a descriptive fixed subset while $f^t(A)$ is a descriptive fixed subset for some $t\neq 1$.
\item 
$A$ is an almost descriptive fixed subset of $f$, provided $\Phi(f(A))=\Phi(A)$  or  $\Phi(f(A)) \ \delta_{\Phi} \   \Phi(A)$ so that $f(A) \ \dcap \ A \neq \emptyset$ by Lemma~\ref{lemma:dP2converse}.
\end{compactenum}	
\end{definition}

\begin{theorem}\label{thm:amenableRibbons}
Let $(K,\delta_{\Phi})$ be a descriptive proximity space over a CW space $K$ with probe function $\Phi: K \to \mathbb{R}^n$, ribbon $\rb A\in 2^K$, and  $f : (K,\delta_{\Phi})\to (K,\delta_{\Phi})$ a descriptive proximally continuous map such that $f(\rb A) \ \dcap \ \rb A \neq \emptyset$.  Then $\rb A, f(\rb A)$ are amiable fixed sets.
\end{theorem}
\begin{proof}
Immediate from Def.~\ref{def:desfixed}.
\end{proof} 

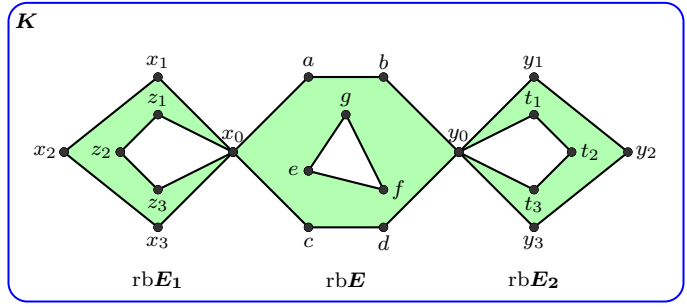
\begin{figure}[!ht]
	\centering
	\begin{pspicture}
	(-4.5,-2)(4.5,2)
	\psframe[linewidth=0.75pt,linearc=0.25,cornersize=absolute,linecolor=blue](-4.5,-2)(4.5,2)

	\psline*[linestyle=solid,linecolor=green!30]%
	(-3.75,0)(-2.5,1)(-1.5,0)(-2.5,-1)(-3.75,0)
	\psline[linestyle=solid,linecolor=black]%
	(-3.75,0)(-2.5,1)(-1.5,0)(-2.5,-1)(-3.75,0)
	\psdots[dotstyle=o,dotsize=3.5pt,linewidth=1.2pt,linecolor=black,fillcolor=black!80]%
	(-3.75,0)(-2.5,1)(-1.5,0)(-2.5,-1)(-3.75,0)
	\psline*[linestyle=solid,linecolor=white]%
	(-3.0,0)(-2.5,0.5)(-1.5,0)(-2.5,-0.5)(-3.0,0)
	\psline[linestyle=solid,linecolor=black]%
	(-3.0,0)(-2.5,0.5)(-1.5,0)(-2.5,-0.5)(-3.0,0)
	\psdots[dotstyle=o,dotsize=3.5pt,linewidth=1.2pt,linecolor=black,fillcolor=black!80]%
	(-3.0,0)(-2.5,0.5)(-1.5,0)(-2.5,-0.5)(-3.0,0)

	\psline*[linestyle=solid,linecolor=green!30]%
	(-1.5,0)(-0.5,1.0)(0.5,1.0)(1.5,0)(0.5,-1.0)(-0.5,-1.0)(-1.5,0)
	\psline[linestyle=solid,linecolor=black]%
	(-1.5,0)(-0.5,1.0)(0.5,1.0)(1.5,0)(0.5,-1.0)(-0.5,-1.0)(-1.5,0)
	\psdots[dotstyle=o,dotsize=3.5pt,linewidth=1.2pt,linecolor=black,fillcolor=black!80]%
	(-1.5,0)(-0.5,1.0)(0.5,1.0)(1.5,0)(0.5,-1.0)(-0.5,-1.0)(-1.5,0)
	\psline*[linestyle=solid,linecolor=white]%
	(-0.5,-0.25)(0,0.5)(0.5,-0.5)(-0.5,-0.25)
	\psline[linestyle=solid,linecolor=black]%
	(-0.5,-0.25)(0,0.5)(0.5,-0.5)(-0.5,-0.25)
	\psdots[dotstyle=o,dotsize=3.5pt,linewidth=1.2pt,linecolor=black,fillcolor=black!80]%
	(-0.5,-0.25)(0,0.5)(0.5,-0.5)(-0.5,-0.25)

	\psline*[linestyle=solid,linecolor=green!30]%
	(1.5,0)(2.5,1.0)(3.75,0)(2.5,-1.0)(1.5,0)
	\psline[linestyle=solid,linecolor=black]%
	(1.5,0)(2.5,1.0)(3.75,0)(2.5,-1.0)(1.5,0)
	\psdots[dotstyle=o,dotsize=3.5pt,linewidth=1.2pt,linecolor=black,fillcolor=black!80]%
	(1.5,0)(2.5,1.0)(3.75,0)(2.5,-1.0)(1.5,0)
	\psline*[linestyle=solid,linecolor=white]%
	(1.5,0)(2.5,0.5)(3.0,0)(2.5,-0.5)(1.5,0)
	\psline[linestyle=solid,linecolor=black]%
	(1.5,0)(2.5,0.5)(3.0,0)(2.5,-0.5)(1.5,0)
	\psdots[dotstyle=o,dotsize=3.5pt,linewidth=1.2pt,linecolor=black,fillcolor=black!80]%
	(1.5,0)(2.5,0.5)(3.0,0)(2.5,-0.5)(1.5,0)

	\rput(-4.25,1.75){\footnotesize $\boldsymbol{K}$}
	\rput(-2.5,-1.7){\footnotesize $\boldsymbol{\rb E_1}$}
	\rput(0,-1.7){\footnotesize $\boldsymbol{\rb E}$}
	\rput(2.5,-1.7){\footnotesize $\boldsymbol{\rb E_2}$}
	\rput(-1.5,0.22){\footnotesize $x_0$}
	\rput(-2.5,1.2){\footnotesize $x_1$}
	\rput(-4,0){\footnotesize $x_2$}
	\rput(-2.5,-1.2){\footnotesize $x_3$}
	\rput(-2.5,0.7){\footnotesize $z_1$}
	\rput(-3.25,0){\footnotesize $z_2$}
	\rput(-2.5,-0.7){\footnotesize $z_3$}
	\rput(-0.5,1.2){\footnotesize $a$}
	\rput(0.5,1.2){\footnotesize $b$}
	\rput(-0.5,-1.2){\footnotesize $c$}
	\rput(0.5,-1.2){\footnotesize $d$}
	\rput(-0.7,-0.25){\footnotesize $e$}
	\rput(0.7,-0.5){\footnotesize $f$}
	\rput(0,0.7){\footnotesize $g$}
	\rput(1.5,0.22){\footnotesize $y_0$}
	\rput(2.5,1.2){\footnotesize $y_1$}
	\rput(4,0){\footnotesize $y_2$}
	\rput(2.5,-1.2){\footnotesize $y_3$}
	\rput(2.5,0.7){\footnotesize $t_1$}
	\rput(3.25,0){\footnotesize $t_2$}
	\rput(2.5,-0.7){\footnotesize $t_3$}
	\end{pspicture}
	\caption{A cell complex $K$ with three ribbons $\rb E_1$, $\rb E$, and $\rb E_2$.}
	\label{fig:fixedribbon}
\end{figure} 

\begin{definition}\label{def:des}
	Let $(X,\delta_{\Phi})$ be a descriptive proximity space with a probe function $\Phi: X \to \mathbb{R}^n$  and $A, B\in 2^X$.  Then $A$ and $B$ are said to be descriptively proximal, provided $\Phi(A)=\Phi(B)$. In that case, we write 
\[
A \ \underset{\mbox{des}}{=} \ B.
\] 
\end{definition}

Note  that 	if $(X,\delta_{\Phi})$ is a descriptive proximity space, $A\in 2^X$, and  $f : (X,\delta_{\Phi})\to (X,\delta_{\Phi})$ is a  descriptive proximally continuous map,  then\\

\begin{compactenum}
	\item $A=B$ implies  $A \ \underset{\mbox{des}}{=} \ B$,  
	\item $f(A)  \ \underset{\mbox{des}}{=} \ A$, provided $A$ is a descriptively fixed subset of $f$,  and 
	\item $f^t(A) \ \underset{\mbox{des}}{=} \ A$ for some positive integer $t\neq 1$, provided $A$ is an eventual descriptive fixed subset of $f$. 
\end{compactenum}

\begin{remark}
Let $K$ be a CW space and let $(K,\delta_{\Phi})$ be a descriptive proximity space with a probe function $\Phi: 2^K \to \mathbb{R}^n$  and ribbons $\rb A, \rb B\in 2^K$.  Then $\rb A$ and $\rb B$ are said to be descriptively proximal, provided $\Phi(\rb A)=\Phi(\rb B)$.
Ribbons $\rb A, \rb B$ are amiable, provided $\rb A \ \underset{\mbox{des}}{=} \ \rb B$.
\end{remark}

\begin{example}

Let $K$ be the CW space in Fig.~\ref{fig:fixedribbon} and let $(K,\delta_{\Phi})$ be a descriptive proximity space, $\rb E_1, \rb E_2\in 2^{K}$ and let  $f : (2^{K},\delta_{\Phi})\to (2^{K},\delta_{\Phi})$ be a  descriptive proximally continuous map such that
\begin{align*}
f(\rb E_1) &= \Phi(\rb E_1), f(\rb E_2) = \Phi(\rb E_2),\ \mbox{with}\\
\Phi(\rb E_1),  &= \beta_{\alpha}(\rb E_1) = \Phi(\rb E_2) = \beta_{\alpha}(\rb E_2) = 1\ \mbox{(from Theorem~\ref{thm:bridgesAndIntersectingCycles})}.
\end{align*}
Consequently, $\rb E_1\ \dcap\ \rb E_2$.  Hence, from Theorem~\ref{thm:amenableRibbons}, $\rb E_1, \rb E_2$ are amenable ribbons.
Further, from Def.~\ref{def:des}, $\rb E_1 \ \underset{\mbox{des}}{=} \ \rb E_2$, since $\Phi(\rb E_1) = \Phi(\rb E_2)$.
\textcolor{blue}{\Squaresteel}
\end{example}


\begin{example}
	Let $K$ be a cell complex with three ribbons $\rb E_1$, $\rb E$, and $\rb E_2$ shown in Figure~\ref{fig:fixedribbon} equipped with the relations 
	\begin{equation*}
	\begin{split}
	&A \ \delta \ B  \quad : \Leftrightarrow \quad  A \ \mbox{and} \ B \  \mbox{have at least one vertex in common} \\
	\mbox{and} \\
	&A \ \delta_{\Phi} \ B \quad : \Leftrightarrow \quad  \beta_{\alpha}(A) = \beta_{\alpha}(B)
	\end{split}
	\end{equation*}
	for a pair of cell complexes $A,  B \in 2^K$. In that case $K$ is both a proximity space and  a descriptive proximity space and we denote it by the triple  $(K, \delta, \delta_{\Phi})$.
	Let $f$ be  a function 
	\[ f: (K,\delta, \delta_{\Phi}) \to (K,\delta, \delta_{\Phi}) \]
	defined on the set of the vertexes of  the ribbons $\rb E$, $\rb E_1$, and $\rb E_2$ by 
	\begin{align*}
	x_i &\leftrightarrow y_i \quad \mbox{for} \ i=0,1,2,3 \\
	z_j &\leftrightarrow t_j \quad \mbox{for} \ j=1,2,3 \\
	a &\leftrightarrow b \quad \mbox{and} \quad c \leftrightarrow d \\
	g &\mapsto e \mapsto f \mapsto g.
	\end{align*}
	Then $f$ is both proximal continuous and descriptive proximaly continuous. \\
	Notice that $\rb E$ is a fixed subset of $f$ while $\rb E_1$ and $\rb E_2$ are eventually fixed subsets of $f$, since they are not fixed subsets of $f$ and  
	\[ f^2(\rb E_i)=\rb E_i\]
	for $i=1, 2$. Also, $\rb E_1$ and $\rb E_3$ are descriptive fixed subsets of $f$, since 
	\[ \Phi(\rb E_1) = \Phi(f(\rb E_1)) = \Phi(\rb E_2)\]
	and 
	\[  \Phi(\rb E_2) = \Phi(f(\rb E_2)) = \Phi(\rb E_1). \mbox{\textcolor{blue}{\Squaresteel}}\]	
\end{example}

\section{Conjugacy between proximal descriptively continuous maps}
This section introduces proximal conjugacy between two dynamical systems, which is an easy extension of the topological conjugacy~\cite[\S 8.1,p. 243]{AdamsFranzosa:2007}.
Proximal conjugacy is akin to strongly amenable groups in which each of its proximal topological actions has a fixed point~\cite{FrishTamuzFerdowsi2019strongAmenability}.  Let $\sum$ denote either a semigroup or a group.  And let $m(\sum)$ be the set of bounded, real-valued functions $\theta$ on $\sum$ for which
\[
\norm{\theta} = lub_{x\in \sum}\abs{\theta(x)}.
\]
A {\bf mean} $\mu$ on $m(\sum)$ is an element of the $m(\sum)*$ (in the conjugate space $B*$~\cite[p.510]{Day1957amenableSemigroups}) such that, for each $x\in m(\sum)$, we have
\[
glb_{x\in \sum}\theta(x) \leq \mu(x) \leq lub_{x\in \sum}\theta(x).
\]

An element of $\mu$ of $m(\sum)*$ is \emph{left}[\emph{right}] invariant, provided 
\[
\mu(\ell_{\sigma}x) = \mu(x)[\mu(r_{\sigma}x) = 
mu(x)]\ \mbox{for all}\ x\in m(\sum), \sigma\in \sum.
\]
$\mbox{}$\\

\begin{definition}~\cite[p.515]{Day1957amenableSemigroups} $\mbox{}$\\
A semigroup (also group) $\sum$ is amenable, provided there is a mean $\mu$ on $m(\sum)$, which is both left and right invariant.
\end{definition} 

We know that a ribbon with bridge edges attached between its cycles has a free fg Abelian group representation.  This leads to the following result as a consequence of results in M.M. Day's 1957 and 1961 papers.

\begin{theorem}\label{thm:amenableRibbonGroup}
Every group representation of a ribbon is amenable.
\end{theorem}
\begin{proof}
This result is a direct consequence of Theorem 2 from Day~\cite{Day1961amenableSemigroups} and (I)~\cite[p.516]{Day1957amenableSemigroups}, since, by construction, every free fg Abelian group $G$ representation of a ribbon is finite and, by definition, an Abelian semigroup.
\end{proof}

A direct consequence of Theorem~\ref{thm:amenableRibbonGroup} and a result of a generalization of the Kakutani-Markov Theorem~\cite{Day1961amenableSemigroups} is that each amenable ribbon has a fixed point.

\begin{definition}
Two  proximal continuous maps $f: (X,\delta_1) \to (X,\delta_1) $ and $g: (Y,\delta_2) \to (Y,\delta_2)$ are said to be proximal  conjugates, provided there exists a proximal  isomorphism $h: (X,\delta_1)  \to (Y,\delta_2)$ such that $g\circ h = h\circ f$. The function $h$ is called a \emph{proximal conjugacy} between $f$ and $g$.
\end{definition}

\begin{theorem}
	Let $h$ be a proximal conjugacy between $f: (X,\delta_1) \to (X,\delta_1) $ and $g: (Y,\delta_2) \to (Y,\delta_2)$. Then for each $A \subseteq X$ and $n\in \mathbb{Z}_+$, we have $h(f^n(A))=g^n(h(A))$.  
\end{theorem}
\begin{proof}
	The proof follows from the induction on $n$.
\end{proof}

\begin{definition}\label{def:proximalDescriptiveContinuousMaps}
Two  proximal descriptive continuous maps $f: (X,\delta_{\Phi_1}) \to (X,\delta_{\Phi_1}) $ and $g: (Y,\delta_{\Phi_2}) \to (Y,\delta_{\Phi_2})$ are said to be proximal descriptive conjugates, provided there exists a proximal descriptive isomorphism $h: (X,\delta_{\Phi_1})  \to (Y,\delta_{\Phi_2})$ such that $g\circ h(A) \underset{\mbox{des}}{=} h\circ f(A)$ for any $A \in 2^X$. The function $h$ is called a proximal descriptive conjugacy between $f$ and $g$. 
\end{definition}

\begin{remark}
	We see from the definition of a proximal descriptive conjugacy that $g\circ h(A)$ and $h\circ f(A)$ may not be equal but we have 
	\[ \Phi_2(g\circ h(A))= \Phi_2(h\circ f(A)) \]
	for $A\in 2^X$. Moreover $g\circ h(A) \underset{\mbox{des}}{=} h\circ f(A)$ implies  $g(A) \underset{\mbox{des}}{=} h\circ f\circ h^{-1}(A)$  and  $f(A) \underset{\mbox{des}}{=} h^{-1}\circ g\circ h(A)$, so that we have the following commutative diagrams.  
\end{remark}

\begin{tikzcd}
	&  & \Phi_1(f(A))=\Phi_1(h^{-1}gh(A))  & \\
	A \arrow[r, mapsto, red, "f"]  \arrow[dr, mapsto, blue, "h"] 
	& f(A)   \arrow[ur, mapsto, red, "\Phi_1"]  & &   h^{-1}gh(A) 
	\arrow[ul, mapsto, blue, "\Phi_1"']  \\ 
	& h(A) \arrow[r, mapsto, blue, "g"]
	&gh(A)     \arrow[ur, mapsto, blue, "h^{-1}"] & 
\end{tikzcd}

\vspace*{1cm}

\begin{tikzcd}
	 & h^{-1}(C)  \arrow[r, mapsto, blue, "f"]  & fh^{-1}(C)    \arrow[dr, mapsto, blue, "h"] & \\
	 C   \arrow[ur, mapsto, blue, "h^{-1}"] \arrow[r, mapsto, red, "g"] & g(C)    \arrow[dr, mapsto, red, "\Phi_2"]   &  & hfh^{-1}(C)   \arrow[dl, mapsto, blue, "\Phi_2"] \\
	&  & \Phi_2(g(C))=\Phi_2(hfh^{-1}(C))   & 
\end{tikzcd}

\begin{remark}
For proximal descriptive conjugates $f: (X,\delta_{\Phi_1}) \to (X,\delta_{\Phi_1}) $ and $g: (Y,\delta_{\Phi_2}) \to (Y,\delta_{\Phi_2})$, Def.~\ref{def:proximalDescriptiveContinuousMaps} tells us that for $A\subseteq X$ and $C\subseteq Y$, we have 
\begin{align*} 
\Phi_2(g\circ h(A)) &=  \Phi_2(h\circ f(A)),  \\
\Phi_1(f\circ h^{-1}(C)) &= \Phi_1(h^{-1}\circ g(C)).\qquad \mbox{\textcolor{blue}{\Squaresteel}}
\end{align*}
\end{remark}

Note that  if  $h$ is a proximal descriptive conjugacy between $f: (X,\delta_{\Phi_1}) \to (X,\delta_{\Phi_1}) $ and $g: (Y,\delta_{\Phi_2}) \to (Y,\delta_{\Phi_2})$, then $A \ \underset{\mbox{des}}{=} \  B$  implies $h(A) \ \underset{\mbox{des}}{=} \  h(B)$ for $A, B \in 2^X$.\\
\vspace{3mm}


\begin{theorem} \label{thm:desconj}
Let $h$ be a proximal descriptive conjugacy between $f: (X,\delta_{\Phi_1}) \to (X,\delta_{\Phi_1}) $ and $g: (Y,\delta_{\Phi_2}) \to (Y,\delta_{\Phi_2})$. Then for each $A \in 2^X$ and $n\in \mathbb{Z}_+$, we have $h(f^n(A)) \ \underset{\mbox{des}}{=} \   g^n(h(A))$.  
\end{theorem}
\begin{proof}
The proof follows from the induction on $n$.
\end{proof}

\begin{corollary}\label{cor:proxDescripConjugacy}
		Let $h$ be a proximal descriptive conjugacy between $f: (X,\delta_{\Phi_1}) \to (X,\delta_{\Phi_1}) $ and $g: (Y,\delta_{\Phi_2}) \to (Y,\delta_{\Phi_2})$.
		
		\begin{compactenum}[a)]
			\item If $A$ is a descriptively fixed subset of $f$, then $h(A)$ is a descriptively fixed subset of $g$.
			\item If $A$ is an eventual descriptively fixed subset of $f$, then $h(A)$ is an eventual  descriptively fixed subset of $g$.
			\item If $A$ is an almost descriptively fixed subset of $f$, then $h(A)$ is an almost  descriptively fixed subset of $g$.
		\end{compactenum}
\end{corollary}
\begin{proof}
	\begin{compactenum}[a)]
		\item  Let $A$ be a descriptively fixed subset of $f$. That is,  $\Phi_1(f(A))= \Phi_1(A)$. In other words, we have $f(A) \ \underset{\mbox{des}}{=} \  A$. 
		Since $h$ is a proximal isomorhism,  $h$ preserves desciptive proximity  $h(f(A)) \ \underset{\mbox{des}}{=} \  h(A)$.  By Theorem~\ref{thm:desconj}, $g(h(A)) \ \underset{\mbox{des}}{=} \   h(A)$ so that $h(A)$ is a  descriptively fixed subset of $g$. \\ 
		\item Let $A$ be an eventual descriptively fixed subset of $f$. That is, $A$ is not a descriptively fixed subset of $f$ but  $\Phi_1(f^n(A))= \Phi_1(A)$ for some positive integer $n>1$. In other words, we have $f^n(A) \ \underset{\mbox{des}}{=} \  A$. 
		Since $h$ is a proximal isomorhism,  $h$ preserves being equal  in a descriptive sense:  $h(f^n(A)) \ \underset{\mbox{des}}{=} \  h(A)$. By Theorem~\ref{thm:desconj}, $g^n(h(A)) \ \underset{\mbox{des}}{=} \  h(A)$. 
		Note that $h(A)$ is not a descriptively fixed subset of $g$ since  $A$ is not a descriptively fixed subset of $f$  and $h$ is an isomorphism. So, $h(A)$ is an eventual descriptively fixed subset of $g$. \\ 
		\item Let $A$ be  an almost descriptively fixed subset of $f$. That is, $f(A) \ \underset{\mbox{des}}{=} \  A$ or     $A \ \delta_{\Phi_1}  \ f(A)$. If $f(A) \ \underset{\mbox{des}}{=} \  A$, then  we are done. Let $A \ \delta_{\Phi_1}  \ f(A)$. Since $h$  is a proximal isomorphism, we have $h(A) \ \delta_{\Phi_2}  \ h(f(A))$. By Theorem~\ref{thm:desconj},  $h(A) \ \delta_{\Phi_2} \ g(h(A))$ so that $h(A)$ is a descriptively fixed subset of $g$.
	\end{compactenum}
\end{proof}

\section{Weak conjugacy between descriptive proximally continuous maps}
This section introduces weak conjugacy between  descriptive proximally continuous maps.

\begin{definition}\label{def:weakProxConjugacy}
	Two  proximally continuous maps $f: (X,\delta_1) \to (X,\delta_1) $ and $g: (Y,\delta_2) \to (Y,\delta_2)$ are said to be weakly proximal conjugates, provided there exists a proximal  isomorphism $h:(X,\delta_1) \to (Y,\delta_2)$ such that for any $A\in 2^X$,  $g\circ h(A)\ \delta_2\ h\circ f(A)$. Note that this also implies that $f\circ h^{-1}(C)\ \delta_1  \ h^{-1}\circ g(C)$ for any $C\in 2^Y$.
	The function $h$ is called a  weakly proximal conjugacy between $f$ and $g$. 
\end{definition}

\begin{theorem}
Let $h$ be a weakly proximal conjugacy between $f: (X,\delta_1) \to (X,\delta_1) $ and $g: (Y,\delta_2) \to (Y,\delta_2)$. Then for each $A\in 2^X$ and $n\in \mathbb{Z}_+$, we have\\ 
$h(f^n(A)) \ \delta_2 \ g^n(h(A))$.  
\end{theorem}
\begin{proof}
The proof follows from the induction on $n$.
\end{proof}

\begin{definition}\label{def:weakProxDescrConjugates}
Two descriptive proximally continuous maps $f: (X,\delta_{\Phi_1}) \to (X,\delta_{\Phi_1}) $ and $g: (Y,\delta_{\Phi_2}) \to (Y,\delta_{\Phi_2})$ are said to be weakly proximal descriptive conjugates, provided there exists a proximal descriptive isomorphism $h: (X,\delta_{\Phi_1})  \to (Y,\delta_{\Phi_2})$ such that $g\circ h(A)\ \delta_{\Phi_2}\ h\circ f(A)$ for any $A \in 2^X$. Note that this also implies  $f\circ h^{-1}(C)\ \delta_{\Phi_2}\ h^{-1}\circ g(C)$ for any $C\in 2^Y$. The function $h$ is called a weakly proximal descriptive conjugacy between $f$ and $g$. 
\end{definition}

\begin{remark}
For weakly proximal descriptive conjugates $f: (X,\delta_{\Phi_1}) \to (X,\delta_{\Phi_1}) $ and $g: (Y,\delta_{\Phi_2}) \to (Y,\delta_{\Phi_2})$, Def.~\ref{def:weakProxDescrConjugates} and Lemma~\ref{lemma:dP2converse} tell us that for $A\in 2^X$ and $C\in 2^Y$, we have 
\begin{align*} 
g\circ h(A)\ &\dcap\  f\circ h(A)\neq \emptyset, \\
f\circ h^{-1}(C)\ &\dcap\ h^{-1}\circ g(C) \neq \emptyset.\qquad \mbox{\textcolor{blue}{\Squaresteel}}
\end{align*}
\end{remark}

\bibliographystyle{amsplain}
\bibliography{NSrefs}

\end{document}